\documentclass[11pt]{article}
\usepackage{a4}
\usepackage{amssymb}
\usepackage{amsmath}
\usepackage{amsthm}
\usepackage{mathtools}
\usepackage{enumerate}
\usepackage{bbm}
\usepackage[pagebackref,colorlinks,citecolor=blue,linkcolor=blue]{hyperref}

\usepackage{geometry}
\numberwithin{equation}{section}
\theoremstyle{plain}
\newtheorem{theorem}[equation]{Theorem}

\newtheorem{lemma}[equation]{Lemma}
\newtheorem{proposition}[equation]{Proposition}

\theoremstyle{definition}
\newtheorem{definition}[equation]{Definition}

\newtheorem{example}[equation]{Example}

\newtheorem*{openproblem}{Open Problem}

\numberwithin{equation}{section}

\newcommand{\R}{{\mathbb R}}
\newcommand{\N}{{\mathbb N}}

\newcommand{\Om}{\Omega}

\providecommand{\vint}[1]{\mathchoice
          {\mathop{\vrule width 5pt height 3 pt depth -2.5pt
                  \kern -9pt \kern 1pt\intop}\nolimits_{\kern -5pt{#1}}}
          {\mathop{\vrule width 5pt height 3 pt depth -2.6pt
                  \kern -6pt \intop}\nolimits_{\kern -3pt{#1}}}
          {\mathop{\vrule width 5pt height 3 pt depth -2.6pt
                  \kern -6pt \intop}\nolimits_{\kern -3pt{#1}}}
          {\mathop{\vrule width 5pt height 3 pt depth -2.6pt
                  \kern -6pt \intop}\nolimits_{\kern -3pt{#1}}}}

\newcommand{\eps}{\varepsilon}
\newcommand{\loc}{\mathrm{loc}}

\newcommand{\BV}{\mathrm{BV}}
\newcommand{\SBV}{\mathrm{SBV}}

\newcommand{\ch}{\text{\raise 1.3pt \hbox{$\chi$}\kern-0.2pt}}
\newcommand{\mres}{\!\mathbin{\vrule height 1.6ex depth 0pt width
		0.13ex\vrule height 0.13ex depth 0pt width 1.1ex}\!}

\DeclareMathOperator{\capa}{Cap}

\DeclareMathOperator{\dive}{div}
\DeclareMathOperator{\dist}{dist}

\DeclareMathOperator{\Var}{Var}

\begin{document}
\title{On the regularity of the maximal function\\
of a BV function
\footnote{
\hfill \break {\it Keywords\,}: function of bounded variation, Sobolev function,
non-centered maximal function, quasicontinuity, absolute continuity 
}}
\author{Panu Lahti}
%\footnote{
%}
\maketitle

\begin{abstract}
	We show that the non-centered maximal function of a BV function is quasicontinuous.
	We also show that \emph{if} the non-centered maximal functions of an SBV function
	is a BV function, then it is in fact a Sobolev function. Using a recent result 
	of Weigt \cite{Wei}, we are in particular able to show that the non-centered maximal
	function of a set of finite perimeter is a Sobolev function.	
\end{abstract}

\section{Introduction}

An open problem that has attracted significant attention in the past two decades is the so-called
$W^{1,1}$-problem: is the Hardy-Littlewood
maximal function of a Sobolev function $u\in W^{1,1}(\R^d)$ also
(locally) in the $W^{1,1}$-class? Typically also a bound
$\Vert \nabla Mu\Vert_{L^1(\R^d)}\le C\Vert \nabla u\Vert_{L^1(\R^d)}$
is expected to hold.
In the case $1<p<\infty$ the analogous result is known to hold,
as first shown by Kinnunen \cite{Ki}.
The same is true for the non-centered maximal function,
defined by
\[
M u(x):=\sup_{x\in B(z,r)}\,\,\vint{B(z,r)}|u|\,dy,\quad x\in\R^d.
\]
We will work with this non-centered version that tends to have better regularity
than the ordinary Hardy-Littlewood
maximal function, in which only balls centered at $x$ are considered.
Tanaka \cite{Tan} gave a positive answer to the $W^{1,1}$-problem in the case $d=1$.
Generalizing this to higher dimensions has received
significant attention, but results have been achieved only in very special cases.
Luiro \cite{Lui} gave a positive answer to the problem in the case of
radial functions, whereas Aldaz
and P\'erez L\'azaro \cite{APL09} did the same for \emph{block-decreasing} functions.

In fact, in \cite{APL09} the authors considered functions $u\in\BV(\R^d)$ rather than just
$u\in W^{1,1}(\R^d)$.
And apart from the Sobolev regularity, one can consider other
continuity properties of the maximal function, but these are also not well
understood; see \cite{APL07} for some positive results when $d=1$ as well as
counterexamples, and \cite{APL09}
for continuity results for block-decreasing BV functions in general dimensions.

In the current paper we show that in general dimensions and for general $u\in\BV(\R^d)$,
a very natural continuity property, namely quasicontinuity, can be proven for the maximal function.
All definitions will be given in Section \ref{sec:notation}.
Most of the time we will consider the maximal function $M_{\Om}u$ where
one considers balls contained in an open set $\Om\subset \R^d$.
After proving some preliminary results in Section \ref{sec:preliminary results},
we prove the following
quasicontinuity result in Section \ref{sec:quasicontinuity}.

\begin{theorem}\label{thm:quasicontinuity of noncentered maximal function intro}
Let $u\in\BV(\Om)$. Then $M_{\Om} u$ is $1$-quasicontinuous.
\end{theorem}

Utilizing this result, in Section \ref{sec:continuity and Lusin} we study
continuity properties of $M_{\Om}u$ on lines parallel to coordinate axes.
Then in Section \ref{sec:Sobolev property} we examine absolute continuity on lines
and membership in the Sobolev class of the maximal function,
proving the following theorem. We say that a BV function is
a special function of bounded variation, or SBV function, if the variation measure has no Cantor part.

\begin{theorem}\label{thm:BV to Sobo intro}
	Let $u\in\SBV(\R^d)$. If $Mu\in \BV_{\loc}(\R^d)$, then
	$M u\in W_{\loc}^{1,1}(\R^d)$. 
\end{theorem}

In a potential breakthrough toward a solution to the $W^{1,1}$-problem,
Weigt \cite{Wei} has shown very recently
that for a set of finite perimeter $E\subset\Om$, we have
$M_{\Om}\mathbbm{1}_E\in\BV_{\loc}(\Om)$ such that $|D M_{\Om}\mathbbm{1}_E|(\Om)$
is at most a constant times $|D\mathbbm{1}_E|(\Om)$.
We can utilize this result and go a step further to the desired Sobolev regularity
at least in the global case $\Om=\R^d$, as follows.

\begin{theorem}\label{thm:perimeter to Sobo intro}
Let $E\subset \R^d$ be a set of finite perimeter.
Then $M\mathbbm{1}_E\in W^{1,1}_{\loc}(\R^d)$ with
$\Vert \nabla M\mathbbm{1}_E\Vert_{L^1(\R^d)}
\le C_d|D\mathbbm{1}_E|(\R^d)$, where $C_d$ only depends on the dimension $d$.
\end{theorem}

Finally, in Sections \ref{sec:formula for derivative}
and \ref{sec:the 1d case} we study formulas
for the gradient of the maximal function, as well
as some further properties in the case $d=1$.

\paragraph{Acknowledgments.}
The author wishes to thank Julian Weigt for comments and for suggesting the proof of 
Lemma \ref{lem:point in closure of ball}. Part of the research for this paper was done
while the author was employed at the University of Augsburg.

\section{Notation and definitions}\label{sec:notation}

\subsection{Basic notation}

We will always work in the Euclidean space $\R^d$, $d\ge 1$.
We denote the $d$-dimensional Lebesgue measure by $\mathcal L^d$ and the
$s$-dimensional Hausdorff measure by $\mathcal H^{s}$, $s\ge 0$.
We denote the characteristic function of a set $E\subset\R^d$ by $\mathbbm{1}_E$.

We write $B(x,r)$ for an open ball in~$\R^{d}$ with center~$x$
and radius~$r$, that is, $\{y \in \R^d \colon |y-x|<r\}$, and we write $\mathbb S^{d-1}$ 
for the unit sphere in $\R^d$, that is, $\{y \in \R^d \colon |y| =1\}$.
When we consider closed balls, we always specify this by the bar $\overline{B}(x,r)$.

For a function $u$ we write $u_+ \coloneqq \max\{u,0\}$ for its positive part, and 
if it is integrable on some measurable set $D \subset \R^d$ of positive and finite Lebesgue
measure, we write
\[
\vint{D} u(y) \,dy \coloneqq \frac{1}{\mathcal L^d(D)} \int_D u(y) \,dy
\] 
for its mean value on $D$.

We will always denote by
$\Om\subset\R^d$ a nonempty open set.
The Sobolev space $W^{1,1}(\Om)$ consists of functions $u\in L^1(\Om)$
whose first weak partial derivatives $D_k u$, $k=1,\ldots,d$, belong to $L^1(\Om)$.

The Sobolev $1$-capacity of a set $A\subset \R^d$ is defined by
\[
\capa_1(A):=\inf \int_{\R^n} \big(|u|+|Du|\big)\,dy,
\]
where the infimum is taken over Sobolev functions $u\in W^{1,1}(\R^d)$ satisfying
$u\ge 1$ in a neighborhood of $A$. The Sobolev $1$-capacity is countably
subadditive.
Using a cutoff function we find that for every ball $B(x,r)$ with $0<r\le 1$, we have
\begin{equation}\label{eq:capacity ball estimate}
\capa_1(B(x,r))\le C_0 r^{d-1}
\end{equation}
for a constant $C_0$ depending only on $d$.

We say that a function $v$ on $\Om$ (generally we understand functions
to take values in $[-\infty,\infty]$)
is $1$-quasicontinuous if for every $\eps>0$
there exists an open set $G\subset \Om$ such that $\capa_1(G)<\eps$
and $v|_{\Om\setminus G}$ is finite and continuous.

By e.g. \cite[Theorem 4.3, Theorem 5.1]{HaKi} we know that for any $A\subset \R^d$,
\begin{equation}\label{eq:null sets of capa and Hausdorff}
\capa_1(A)=0\quad\textrm{if and only if}\quad\mathcal H^{d-1}(A)=0.
\end{equation}

The non-centered maximal function of a measurable function
$u$ on $\Om$ is defined by
\[
M_{\Om} u(x):=\sup_{x\in B(z,r)\subset \Om}\,\,\vint{B(z,r)}|u|\,dy,\quad x\in\Om.
\]

Sometimes we do not mention $\Om$ and then it is understood that
$\Om=\R^d$, so that $Mu:=M_{\R^d}u$.

For $\ell\in\N$, we denote by
$\mathcal M(\Om;\R^{\ell})$ the Banach space of vector-valued
Radon measures $\mu$, equipped with the
\emph{total variation norm} $|\mu|(\Om)<\infty$, which is defined
relative to the Euclidean norm on $\R^{\ell}$.
By the Riesz
representation theorem, $\mathcal M(\Om;\R^{\ell})$ can be identified
with the dual space of $C_0(\Om;\R^{\ell})$ through the duality
pairing
$\langle\phi,\mu\rangle\coloneqq \int_{\Om} \phi\cdot d\mu\coloneqq \sum_{j=1}^{\ell} 
\int_{\Om} \phi_j \,d\mu_j$.
Thus weak* convergence $\mu_i\overset{*}{\rightharpoondown}\mu$
in $\mathcal M(\Om;\R^{\ell})$ means $\langle\phi,\mu_i\rangle\to \langle\phi,\mu\rangle$ for all $\phi\in C_0(\Om;\R^{\ell})$.

We denote the restriction of a measure $\nu$ to a set $A\subset\R^d$
by $\nu \mres A$, that is,
\[
\nu \mres A(H):=\nu(A\cap H),\quad H\subset \R^d.
\]
For a vector-valued Radon measure $\gamma\in\mathcal M(\Om;\R^{\ell})$
and a positive Radon measure, we
can write the Radon-Nikodym decomposition
\[
\gamma=\gamma^a+\gamma^s=\frac{d\gamma}{d\mu}\,d\mu+\gamma^s
\]
of $\gamma$ with respect to $\mu$,
where $\frac{d\gamma}{d\mu}\in L^1(\Om,\mu;\R^{\ell})$.

\subsection{Functions of bounded variation}\label{sec:BV functions}

The theory of $\BV$ functions presented here can be found e.g. in \cite{AFP},
and we give precise references only for a few key facts.
A function
$u\in L^1(\Omega)$ is a function of bounded variation,
denoted $u\in \BV(\Omega)$, if its distributional derivative
is an $\R^{d}$-valued Radon measure with finite total variation. This means that
there exists a (unique) Radon measure $Du\in \mathcal M(\Om;\R^{d})$
such that for all $\varphi\in C_c^1(\Omega)$, the integration-by-parts formula
\[
\int_{\Omega}u\frac{\partial\varphi}{\partial y_j}\,dy
=-\int_{\Omega}\varphi\,d(Du)_j,\quad j=1,\ldots,d
\]
holds.

If we do not know a priori that a function $u\in L^1_{\loc}(\Om)$
is a BV function, we consider
\begin{equation}\label{eq:definition of total variation}
\Var(u,\Om):=\sup\left\{\int_{\Om}u\dive\varphi\,dy,\,\varphi\in C_c^{1}(\Om),
\,|\varphi|\le 1\right\}.
\end{equation}
If $\Var(u,\Om)<\infty$, then the Radon measure $Du$ exists and $\Var(u,\Om)=|Du|(\Om)$
by the Riesz representation theorem, and $u\in\BV(\Om)$ provided that $u\in L^1(\Om)$.
If $E\subset\R^d$ with $\Var(\mathbbm{1}_E,\R^d)<\infty$, we say that $E$ is a set of finite perimeter.

A fact that we will use many times is that if $u\in\BV_{\loc}(\Om)$, then also $|u|\in\BV_{\loc}(\Om)$
with $|D|u||(\Om)\le |Du|(\Om)$.

Let $u$ be a function on $\Om$.
We say that $x\in\Om$ is a Lebesgue point of $u$ if
\begin{equation}\label{eq:Lebesgue point}
\lim_{r\to 0}\,\vint{B(x,r)}|u(y)-\widetilde{u}(x)|\,dy=0
\end{equation}
for some $\widetilde{u}(x)\in\R$. We denote by $S_u\subset\Om$ the set where
this condition fails and call it the \emph{approximate discontinuity set}.

Given $\nu\in \mathbb S^{d-1}$, we define the half-balls
\begin{align*}
B_{\nu}^+(x,r)\coloneqq \{y\in B(x,r)\colon \langle y-x,\nu\rangle>0\},\\
B_{\nu}^-(x,r)\coloneqq \{y\in B(x,r)\colon \langle y-x,\nu\rangle<0\}.
\end{align*}
We say that $x\in \Om$ is an approximate jump point of $u$ if there exist
$\nu\in \mathbb S^{d-1}$ and distinct numbers $u^+(x),u^-(x)\in\R$ such that
\begin{equation}\label{eq:jump value 1}
\lim_{r\to 0}\,\vint{B_{\nu}^+(x,r)}|u(y)-u^+(x)|\,dy=0
\end{equation}
and
\[
\lim_{r\to 0}\,\vint{B_{\nu}^-(x,r)}|u(y)-u^-(x)|\,dy=0.
\]
The set of all approximate jump points is denoted by $J_u$.
We have that $\mathcal H^{d-1}(S_u\setminus J_u)=0$, see \cite[Theorem 3.78]{AFP}.

The lower and upper approximate limits of a function $u$
are defined respectively by
\[
u^{\wedge}(x)\coloneqq 
\sup\left\{t\in\R\colon \lim_{r\to 0}\frac{\mathcal L^d(B(x,r)\cap\{u<t\})}{\mathcal L^d(B(x,r))}=0\right\}
\]
and
\[
u^{\vee}(x)\coloneqq 
\inf\left\{t\in\R\colon \lim_{r\to 0}\frac{\mathcal L^d(B(x,r)\cap\{u>t\})}{\mathcal L^d(B(x,r))}=0\right\}.
\]

Note that for all $x\in \Om\setminus S_u$, we have
$\widetilde{u}(x)=u^{\wedge}(x)=u^{\vee}(x)$. Also, for all
$x\in J_u$, we have
$u^{\wedge}(x)=\min\{u^{-}(x),u^+(x)\}$ and $u^{\vee}(x)=\max\{u^{-}(x),u^+(x)\}$.

We write the Radon-Nikodym decomposition of the variation measure of $u$ into the absolutely continuous and singular parts as $Du=D^a u+D^s u$. Furthermore, we define the Cantor and jump parts of $Du$ by
\begin{equation}\label{eq:Dju and Dcu}
D^c u\coloneqq  D^s u\mres (\Om\setminus S_u),\qquad D^j u\coloneqq D^s u\mres J_u.
\end{equation}
Since $\mathcal H^{d-1}(S_u\setminus J_u)=0$ and $|Du|$ vanishes on
$\mathcal H^{d-1}$-negligible sets, we get the decomposition
\[
Du=D^a u+ D^c u+ D^j u.
\]
We say that $u\in\BV(\Om)$ is a special function of bounded variation, and
denote $u\in\SBV(\Om)$, if $|D^c u|(\Om)=0$.

\subsection{One-dimensional sections of $\BV$ functions}\label{subsec:one dimensional sections}\label{subsec:1d sections}

For basic results in the one-dimensional case $d=1$
(with slightly different notation from ours), see \cite[Section 3.2]{AFP}.
In this setting, given an open set $\Om\subset\R$ and $u\in\BV_{\loc}(\Om)$,
we have $J_u=S_u$, $J_u$ is at most countable,
and $Du (\{x\})=0$ for every $x\in\Om\setminus J_u$.
For every $x,\widetilde{x}\in \Om$ in a connected component of $\Om$, we have
\begin{equation}\label{eq:fundamental theorem of calculus for BV}
|u^{\vee}(\widetilde{x})-u^{\vee}(x)|\le |Du|([x,\widetilde{x}]).
\end{equation}
Thus at every point outside $S_u$, the pointwise representative
$u^{\wedge}=u^{\vee}=\widetilde{u}$ is continuous.
Moreover, $u^{\vee}$ is upper semicontinuous.

In $\R^d$, denote by $\pi\colon\R^d\to \R^{d-1}$ the orthogonal projection onto $\R^{d-1}$:
for $x=(x_1,\ldots,x_d)\in\R^d$,
\[
\pi((x_1,\ldots,x_d)):=(x_1,\ldots,x_{d-1}).
\] 
Denote the standard basis vectors by $e_k$, $k=1,\ldots,d$.
For an open set $\Om\subset\R^d$,
and $u\in\BV(\Om)$, denote $D_k u\coloneqq \langle Du,e_k\rangle$.
For any fixed $k\in\{1,\ldots,d\}$ --- for simplicity we can assume $k=d$ ---
for every $z\in\pi(\Om)$ we denote the slices of $\Om$ at $(z,0)$ in $e_d$-direction by
\[
\Om_z\coloneqq \{t\in\R\colon (z,t e_d) \in \Om\}.
\]
We also denote
$u_z(t)\coloneqq u(z,t)$ for $z\in\pi(\Om)$ and $t\in \Om_z$.
We know that for $\mathcal L^{d-1}$-almost
every $z\in\pi(\Om)$, we have
$u_z\in \BV(\Om_z)$ (see \cite[Theorem 3.103]{AFP})
and also, if $u\in\SBV(\Om)$, then $u_z\in \SBV(\Om_z)$
(see \cite[Eq. (3.108)]{AFP}).
On the other hand, if $u_z$ is absolutely continuous for almost every $z\in \pi(\Om)$,
and similarly in the other coordinate directions,
then $u\in W^{1,1}_{\loc}(\Om)$.
Finally, for $\mathcal L^{d-1}$-almost every $z\in\pi(\Om)$ it holds that
\begin{equation}\label{eq:sections and jump sets}
S_{u_z}=(S_u)_z\quad\textrm{and}\quad
(\widetilde{u})_z(t)=\widetilde{u_z}(t)\ 
\textrm{for every }t\in \R\setminus S_{u_z},
\end{equation}
see \cite[Theorem 3.108]{AFP}.

\section{Preliminary results}\label{sec:preliminary results}

In this section we record and prove some preliminary results.
Let $\Om\subset\R^d$ always denote an arbitrary nonempty open set.

The following fact is generally well known and used e.g. in \cite{Lui}.
It simply says that in the definition of the non-centered maximal function,
the supremum can be taken over balls whose \emph{closure} contains the point $x$.
For the maximal function $M_{\Om}u$, this fact is not as trivial
as it is for the global version $Mu=M_{\R^d}u$, so we give a short proof.

\begin{lemma}\label{lem:point in closure of ball}
	Let $u$ be a measurable function on $\Om$. Then we have
	\[
	M_{\Om} u(x)=\sup_{x\in \overline{B}(z,r),\,B(z,r)\subset \Om}\,\,
	\vint{B(z,r)}|u|\,dy,\quad x\in\Om.
	\]
\end{lemma}
\begin{proof}
	Consider a ball $B(z,r)\subset \Om$ and a point $x\in \Om\cap \partial B(z,r)$.
	For some $\delta>0$, we have $B(x,\delta)\subset \Om$. Let $0<\eps<1/3$.
	The ball $B(z+\eps(x-z),(1-\eps)r)$ is contained in $B(z,r)$.
	Then clearly for sufficiently small $t\in (0,\eps)$, the ball
	\[
	B(z+(\eps+t)(x-z),(1-\eps)r)
	\]
	is contained in $B(z,r)\cup B(x,\delta)\subset \Om$,
	and contains $x$ and contains $B(z,(1-3\eps)r)$.
	Thus (some of the integrals below could be $+\infty$)
	\[
	M_{\Om}u(x)\ge \vint{B(z+(\eps+t)(x-z),(1-\eps)r)}|u|\,dy
	\ge \frac{1}{\mathcal L^d(B(z,r))}\int_{B(z,(1-3\eps)r)}|u|\,dy.
	\]
	Letting $\eps\to 0$, we obtain
	\[
	M_{\Om}u(x)\ge \vint{B(z,r)}|u|\,dy.
	\]
\end{proof}

The following simple property of the non-centered maximal function
is crucial for proving the $1$-quasicontinuity of $M_{\Om} u$.

\begin{proposition}\label{prop:maximal function and upper representative}
	Let $u\in\BV_{\loc}(\Om)$.
	Then $M_{\Om} u(x)\ge u^{\vee}(x)$ for
	every $x\in \Om\setminus (S_u\setminus J_u)$, that is, for
	$\mathcal H^{d-1}$-almost every $x\in\Om$. 
\end{proposition}
\begin{proof}
	We obviously have $M_{\Om}u(x)\ge \widetilde{u}(x)=u^{\vee}(x)$
	for every $x\in \Om\setminus S_u$,
	that is, for Lebesgue points $x$ (recall \eqref{eq:Lebesgue point}).
	Assume then that $x\in J_u$.
	Now $u^{\vee}(x)=\max\{u^{-}(x),u^+(x)\}$ (recall \eqref{eq:jump value 1}).
	Supposing $u^{\vee}(x)=u^+(x)$, we obtain using
	Lemma \ref{lem:point in closure of ball},
	\[
	M_{\Om}u(x)\ge \limsup_{r\to 0}\vint{B(x+(r/2)\nu,r/2)}u\,dy=u^{\vee}(x).
	\]
	The case $u^{\vee}(x)=u^-(x)$ is similar.
	The proof is completed by recalling that $\mathcal H^{d-1}(S_u\setminus J_u)=0$.
\end{proof}

As noted by Aldaz and P\'erez L\'azaro \cite{APL07}, a BV function
need not have any upper semicontinuous representative when $d\ge 2$,
which causes difficulties
since usually such a representative is used in proving the continuity of $M_{\Om}u$.
However, the following quasi-semicontinuity result given in \cite[Theorem 2.5]{CDLP} is a useful
substitute.
Alternatively, see
\cite[Theorem 1.1]{LaSh} and \cite[Corollary 4.2]{L-SA}
for a proof of this result in more general metric spaces.

\begin{theorem}\label{thm:quasisemicontinuity}
	Let $u\in\BV_{\loc}(\Om)$ and $\eps>0$. Then there exists an open set $G\subset\Om$
	such that $\capa_1(G)<\eps$ and $u^{\wedge}|_{\Om\setminus G}$ is finite and lower
	semicontinuous, and $u^{\vee}|_{\Om\setminus G}$ is finite and upper
	semicontinuous.
\end{theorem}

The following weak type estimate is well known and a proof can be found e.g.
in \cite[Lemma 4.3]{KKST}.

\begin{proposition}\label{prop:weak type estimate}
	Let $u\in\BV(\R^d)$. Then we have for every $t>0$ that
	\[
	\capa_1(\{x\in \R^d\colon Mu(x)>t\})
	\le C_1\frac{|Du|(\R^d)}{t},
	\]
	where $C_1$ only depends on $d$.
\end{proposition}

We define auxiliary maximal operators $M_{\Om}^R$ and $M_{\Om,R}$, $R>0$, by
\[
M_{\Om}^R u(x):=\sup_{x\in B(z,r)\subset\Om,\,r< R}\,\,\vint{B(z,r)}|u|\,dy,
\quad x\in\Om,
\]
and
\begin{equation}\label{eq:auxiliary maximal operator}
M_{\Om,R} u(x):=\sup_{x\in B(z,r)\subset\Om,\,r\ge R}\,\,\vint{B(z,r)}|u|\,dy,
\quad x\in\Om.
\end{equation}

Obviously $M_{\Om}u=\max\{M_{\Om}^R u,M_{\Om,R} u\}$.
Again if $\Om=\R^d$, we omit it from the notation.

\begin{proposition}\label{prop:weak type estimate inifinity}
	Let $u\in\BV_{\loc}(\Om)$. Then $\capa_1(\{x\in \Om\colon M_{\Om}u(x)=\infty\})=0.$
\end{proposition}
\begin{proof}
Consider the open sets
\[
\Om_j:=\{x\in\Om\colon \dist(x,\R^d\setminus \Om)>2^{-j}\},\quad j\in\N.
\]
Now $\bigcup_{j=1}^{\infty}\Om_j=\Om$. Choose
cutoff functions $\eta_j\in C_c^{\infty}(\Om)$ with $0\le\eta_j\le 1$
in $\R^d$ and $\eta_j=1$ in $\Om_j$.
Now $\eta_j u\in\BV(\R^d)$ and so by Proposition \ref{prop:weak type estimate},
\begin{align*}
\capa_1(\{x\in \Om_j\colon M_{\Om}^{2^{-2j}}u(x)=\infty\})
\le \capa_1(\{x\in \Om_j\colon M(\eta_{j+1} u)(x)=\infty\})
=0.
\end{align*}
On the other hand, since $u\in L^1(\Om)$, clearly $M_{\Om,2^{-2j}}u(x)<\infty$ for every $x\in \Om$.
In total $M_{\Om}u(x)=\max\{M_{\Om}^{2^{-2j}}u(x),M_{\Om,2^{-2j}}u(x)\}<\infty$
for $\capa_1$-almost every $x\in \Om_j$.
Since $\bigcup_{j=1}^{\infty}\Om_j=\Om$, we obtain the result.
\end{proof}

The following fact is well known; for a proof covering the case
$u\in L^1(\R^d)$ see e.g. \cite[Proposition 3.2]{Lui},
while the case $u\in L^{\infty}(\R^d)$ follows by a slight modification.
This result does not necessarily hold for $M_{\Om}u$ in an open set $\Om$, which is why we formulate
some of the main results of this paper only in the global case $\Om=\R^d$.

\begin{proposition}\label{prop:Lipschitz continuity}
Let $u\in L^1(\R^d)$ (resp. $u\in L^{\infty}(\R^d)$), and let $R>0$. Then $M_{R} u$
is Lipschitz with constant depending only on $d$, $R$, and $\Vert u\Vert_{L^1(\R^d)}$
(resp. $\Vert u\Vert_{L^{\infty}(\R^d)}$).
\end{proposition}

The following result proven in \cite[Lemma 3.5]{L-SA} is our key tool for handling
the exceptional set of quasi (semi)continuity.

\begin{lemma}\label{lem:uniform convergence of G}
	Let $G\subset \R^d$ and $\eps>0$. Then there exists an open set
	$U\supset G$ with $\capa_1(U)\le C_2\capa_1(G)+\eps$ such that
	\[
	\frac{\mathcal L^d(B(x,r)\cap G)}{\mathcal L^d(B(x,r))}\to 0\quad\textrm{as }r\to 0
	\]
	uniformly for $x\in \R^d\setminus U$. Here $C_2$ depends only on $d$.
\end{lemma}

We will need the following version.

\begin{lemma}\label{lem:uniform convergence of u}
	Let $u\in\BV_{\loc}(\Om)$, let $G\subset \Om$,
	and let $\eps>0$. Then there exists an open set
	$U\supset G$ such that $\capa_1(U)\le C_2\capa_1(G)+\eps$ and
	\[
	\frac{1}{\mathcal L^d(B(x,r))}\int_{B(x,r)\cap G}|u|\,dy
	\to 0\quad\textrm{as }r\to 0
	\]
	locally uniformly for $x\in \Om\setminus U$. Here $C_2$ is the same constant
	as in Lemma \ref{lem:uniform convergence of G}.
\end{lemma}

\begin{proof}
We  have $|u|\in \BV_{\loc}(\Om)$, and so we can assume that
$u$ is nonnegative.
By Lemma \ref{lem:uniform convergence of G},
we find an open set $W\supset G$ such that
$\capa_1(W)\le C_2\capa_1(G)+\eps/2$ and
\begin{equation}\label{eq:condition on G}
\frac{\mathcal L^d(B(x,r)\cap G)}{\mathcal L^d(B(x,r))}
\to 0\quad\textrm{as }r\to 0
\end{equation}
uniformly for $x\in \R^d\setminus W$.
Let
\[
\Om_j:=\{x\in B(0,j)\colon \dist(x,\R^d\setminus \Om)>1/j\},\quad j\in\N,
\]
so that $\Om=\bigcup_{j=1}^{\infty}\Om_j$.
Note that we have by the Sobolev embedding
\[
u\in\BV_{\loc}(\Om)\subset L_{\loc}^{d/(d-1)}(\Om)
\subset L^{d/(d-1)}(\Om_{j})\quad\textrm{for every }j\in\N.
\]
Choose numbers $\alpha_j>0$, $j=0,1,\ldots$, such that $\alpha_{j+1}\ge 2\alpha_j$ and
\[
\sum_{j=1}^{\infty}\left(\int_{\Om_{j+1}}(u-\alpha_{j-1})_+^{d/(d-1)}\,dy
\right)^{(d-1)/d}<\frac{\eps}{5^{d} C_0}.
\]
Next take a sequence $\beta_j\searrow 0$, $\beta_j\le 1$, such that still
\begin{equation}\label{eq:beta sum condition}
\sum_{j=1}^{\infty}\frac{1}{\beta_j}
\left(\int_{\Om_{j+1}}(u-\alpha_{j-1})_+^{d/(d-1)}\,dy
\right)^{(d-1)/d}<\frac{\eps}{5^{d} C_0}.
\end{equation}
Define the sets
\[
E_j:=\{x\in \Om\colon u(x)\ge \alpha_j\}.
\]
Then define the sets
\begin{align*}
&A_j:=\Bigg\{x\in \Om_{j}\colon
\frac{1}{\mathcal L^d(B(x,r))}\int_{B(x,r)\cap E_j}u\,dy>\beta_j\\
&\qquad\qquad\qquad\qquad\textrm{ for some }0<r\le 1/5
\textrm{ with }B(x,r)\subset \Om_{j+1}\Bigg\}.
\end{align*}
Consider $j\in\N$ and $x\in A_j$.
For some $0<r_x\le 1/5$, we have $B(x,r_x)\subset\Om_{j+1}$ and
\[
\frac{1}{\mathcal L^d(B(x,r_x))}\int_{B(x,r_x)}(u-\alpha_{j-1})_+\,dy
\ge \frac{1}{2}\frac{1}{\mathcal L^d(B(x,r_x))}\int_{B(x,r_x)\cap E_j}u\,dy
>\frac{\beta_j}{2},
\]
and so by H\"older's inequality
\[
\frac{1}{r_x^{d-1}}\left(\int_{B(x,r_x)}(u-\alpha_{j-1})_+^{d/(d-1)}\,dy\right)^{(d-1)/d}
>\frac{\beta_j}{2}.
\]
Now $\{B(x,r_x)\}_{x\in A_j}$ is a covering of $A_j$.
By the $5$-covering theorem, we find a countable collection of pairwise
disjoint balls $\{B(x_k,r_k)\}_{k=1}^{\infty}$ such that
$A_j\subset \bigcup_{k=1}^{\infty}B(x_k,5r_k)$.
Now we have by \eqref{eq:capacity ball estimate}, and by using
the triangle inequality for the $L^{q/(q-1)}$-norm,
\begin{align*}
\capa_1(A_j)
&\le \sum_{k=1}^{\infty}\capa_1(B(x_k,5r_k))\\
&\le 5^{d-1}C_0\sum_{k=1}^{\infty}r_k^{d-1}\\
&\le \frac{2\times 5^{d-1}C_0}{\beta_j}\sum_{k=1}^{\infty}
\left(\int_{B(x_k,r_k)}(u-\alpha_{j-1})_+^{d/(d-1)}\,dy\right)^{(d-1)/d}\\
&\le \frac{2\times 5^{d-1}C_0}{\beta_j}
\left(\int_{\Om_{j+1}}(u-\alpha_{j-1})_+^{d/(d-1)}\,dy\right)^{(d-1)/d}.
\end{align*}
Now recalling \eqref{eq:beta sum condition}, we get
\[
\capa_1\left(W\cup\bigcup_{j=1}^{\infty}A_j\right)
\le \capa_1(W)+\sum_{j=1}^{\infty}\capa_1(A_j)
<C_2\capa_1(G)+\frac{\eps}{2}+\frac{\eps}{2}.
\]
Finally, take an open set
\[
U\supset W\cup\bigcup_{j=1}^{\infty}A_j
\]
with $\capa_1(U)<C_2\capa_1(G)+\eps$.

Fix $\delta>0$.
Take $j_0\in\N$ sufficiently large that
$j_0\ge 1/\delta$ and $\beta_{j_0}<\delta/2$.
Using \eqref{eq:condition on G}, take $0<R\le 1/5$ such that
\[
\frac{\mathcal L^d(G\cap B(x,r))}{\mathcal L^d(B(x,r))}<\frac{\delta}{2\alpha_{j_0}}
\quad\textrm{for all }x\in \R^d\setminus U\textrm{ and }
0<r\le R.
\]
Thus for all
\[
x\in \{y\in B(0,1/\delta)\colon \dist(y,\R^d\setminus \Om)>\delta\}
\setminus U\subset \Om_{j_0}\setminus U
\]
and $0<r\le \min\{R,\dist(\Om_{j_0},\R^d\setminus \Om_{j_0+1})\}$,
we have
\begin{align*}
&\frac{1}{\mathcal L^d(B(x,r))}\int_{G\cap B(x,r)}u\,dy\\
&\quad\le \frac{1}{\mathcal L^d(B(x,r))}\int_{G\cap B(x,r)\setminus E_{j_0}}u\,dy
+\frac{1}{\mathcal L^d(B(x,r))}\int_{B(x,r)\cap E_{j_0}}u\,dy\\
&\quad<\frac{\delta}{2\alpha_{j_0}}\alpha_{j_0}+\beta_{j_0}
\le \frac{\delta}{2}+\frac{\delta}{2}
=\delta.
\end{align*}
Since $\delta>0$ was arbitrary, this proves the local uniform convergence in
the set $\Om\setminus U$.
\end{proof}

\section{Quasicontinuity}\label{sec:quasicontinuity}

In this section we prove that the non-centered maximal function of
a BV function is $1$-quasi\-continuous.
As before, $\Om\subset \R^d$ is an arbitrary nonempty open set.

The following theorem is Theorem \ref{thm:quasicontinuity of noncentered maximal function intro}
in a slightly more general form.

\begin{theorem}\label{thm:quasicontinuity of noncentered maximal function}
	Let $u\in\BV_{\loc}(\Om)\cap L^{1}(\Om)$ or $u\in \BV_{\loc}(\R^d)$.
	Then $M_{\Om} u$ is $1$-quasicontinuous.
\end{theorem}

\begin{proof}
	First assume that $u\in\BV_{\loc}(\Om)\cap L^{1}(\Om)$.
	Then also $|u|\in \BV_{\loc}(\Om)\cap L^{1}(\Om)$, and so we can assume that $u$ is nonnegative.
	
	Fix $\eps>0$.
	By Theorem \ref{thm:quasisemicontinuity} we find
	an open set $G\subset \Om$ such that
	$\capa_1(G)<\eps/C_2$ and $u^{\vee}|_{\Om\setminus G}$ is upper semicontinuous.
	Since $\capa_1(S_u\setminus J_u)=\mathcal H^{d-1}(S_u\setminus J_u)=0$
	(recall \eqref{eq:null sets of capa and Hausdorff})
	and
	\[
	\capa_1(\{x\in\Om:\,M_{\Om} u(x)=\infty\})=0
	\]
	by Proposition \ref{prop:weak type estimate inifinity}, we can assume that
	$G\supset \{x\in\Om:\,M_{\Om} u(x)=\infty\}\cup (S_u\setminus J_u)$. 
	Then by Lemma \ref{lem:uniform convergence of u}, we can take an open set
	$U\supset G$ such that $\capa_1(U)<\eps$ and
	\begin{equation}\label{eq:uniform convergence G}
	\frac{1}{\mathcal L^d(B(x,r))}\int_{B(x,r)\cap G}u\,dy
	\to 0\quad\textrm{as }r\to 0
	\end{equation}
	locally uniformly for $x\in \Om\setminus U$.
	Since $M_{\Om}u|_{\Om\setminus U}$ is finite and lower semicontinuous, it is
	sufficient to prove upper semicontinuity.
	Fix $x\in\Om\setminus U$.
	Take a sequence $x_j\to x$, $x_j\in \Om\setminus U$, such that
	\[
	\lim_{j\to\infty}M_{\Om} u(x_j)=\limsup_{\Om\setminus U\ni y\to x}M_{\Om} u(y).
	\]
	(At this stage we cannot exclude the possibility that the $\limsup$ is $\infty$.)
	Now we only need to show that
	$M_{\Om} u(x)\ge \lim_{j\to\infty}M_{\Om} u(x_j)$.
	
	We find ``almost optimal'' balls $B(x_j^*,r_j)$ in the sense that 
	\begin{equation}\label{eq:choice of almost optimal balls}
	\lim_{j\to\infty}M_{\Om}u(x_j)=\lim_{j\to\infty}\,\vint{B(x_j^*,r_j)}u\,dy,
	\end{equation}
	with $x_j\in B(x_j^*,r_j)\subset\Om$.
	Since $u\in L^1(\Om)$, we can assume that the radii $r_j$ are uniformly bounded.
	Now we consider two cases.
	
	\textbf{Case 1.} Suppose that by passing to a subsequence (not relabeled),
	we have $r_j\to 0$. Fix $\delta>0$.
	By the upper semicontinuity of $u^{\vee}|_{\Om\setminus G}$, for some
	$r>0$ we have $B(x,r)\subset\Om$ and
	\begin{equation}\label{eq:using upper semicontinuity}
	u^{\vee}(x)\ge \sup_{B(x,r)\setminus G}u^{\vee}-\delta.
	\end{equation}
	Note also that for sufficiently large $j\in\N$, we have
	$B(x_j^*,r_j)\subset B(x,r)$.
	Thus, using Proposition \ref{prop:maximal function and upper representative}
	(recall that $G\supset S_u\setminus J_u$),
	we get for large $j\in\N$
	\begin{align*}
	M_{\Om} u(x)
	&\ge u^{\vee}(x)\\
	&\ge \sup_{B(x,r)\setminus G}u^{\vee}-\delta\quad\textrm{by }\eqref{eq:using upper semicontinuity}\\
	&\ge \frac{1}{\mathcal L^d(B(x_j^*,r_j))}\int_{B(x_j^*,r_j)\setminus G}u\,dy-\delta\quad\textrm{since }B(x_j^*,r_j)\subset B(x,r)\\
	&= \frac{1}{\mathcal L^d(B(x_j^*,r_j))}\int_{B(x_j^*,r_j)}u\,dy-
	\frac{1}{\mathcal L^d(B(x_j^*,r_j))}\int_{B(x_j^*,r_j)\cap G}u\,dy-\delta\\
		&\ge \vint{B(x_j^*,r_j)}u\,dy-
	\frac{2^d}{\mathcal L^d(B(x_j,2r_j))}\int_{B(x_j,2r_j)\cap G}u\,dy-\delta.
	\end{align*}
	Now by \eqref{eq:uniform convergence G} and \eqref{eq:choice of almost optimal balls},
	we get
	\[
	M_{\Om} u(x)\ge \lim_{j\to\infty}M_{\Om} u(x_j)-\delta,
	\]
	so that letting $\delta\to 0$, we obtain the desired inequality.
	
	\textbf{Case 2.}
	The other alternative is that passing to a subsequence (not relabeled),
	we have $r_j\to r\in (0,\infty)$.
	Passing to a further subsequence (not relabeled),
	the vectors $x_j^*-x_j$ (since they have length at most $r_j)$
	converge to some $v\in\R^d$.
	Now for $x^*:=x+v$ we have
	$\mathbbm{1}_{B(x_j^*,r_j)}\to \mathbbm{1}_{B(x^*,r)}$ in $L^1(\R^d)$, with
	$x\in \overline{B}(x^*,r)$ and $B(x^*,r)\subset \Om$.
	In this case we have by Lemma \ref{lem:point in closure of ball} that
	\[
	M_{\Om} u(x)\ge \vint{B(x^*,r)}u\,dy
	=\lim_{j\to\infty}\,\vint{B(x_j^*,r_j)}u\,dy
	=\lim_{j\to\infty}M_{\Om} u(x_j).
	\]
	This completes the proof in the case $u\in\BV_{\loc}(\Om)\cap L^{1}(\Om)$.
	
	Now consider the case $u\in \BV_{\loc}(\R^d)$. The proof is the same,
	except that now we need to consider a third case.
	
	\textbf{Case 3.} Suppose that passing to a subsequence (not relabeled),
	we have $r_j\to \infty$.
	Then since $B(x_j^*,r_j+1)\ni x$ for all large $j\in\N$, we get
	\[
	M u(x)\ge \limsup_{j\to\infty}\vint{B(x_j^*,r_j+1)}u\,dy
	\ge \limsup_{j\to\infty}\vint{B(x_j^*,r_j)}u\,dy
	=\lim_{j\to\infty}M u(x_j).
	\]
	This completes the proof.	
\end{proof}

Aldaz and P\'erez L\'azaro \cite{APL09}
showed that the maximal function of a
\emph{block-decreasing} BV function is continuous at every point outside a
$\mathcal H^{d-1}$-negligible set. Such a continuity property is somewhat
stronger than $1$-quasicontinuity; recall that $\capa_1$ and $\mathcal H^{d-1}$
have the same null sets. The following simple example shows that 
for general BV functions we cannot have such a stronger continuity property,
demonstrating
that quasicontinuity seems to be the correct concept to consider.

\begin{example}
	Take an enumeration of all the points on the plane with rational
	coordinates $\{q_j\}_{j=1}^{\infty}$ and define
	the ``enlarged rationals''
	\[
	E:=\bigcup_{j=1}^{\infty}B(q_j,2^{-j}).
	\]
	Clearly $\mathcal L^2(E)\le \pi/2$.
	By lower semicontinuity and subadditivity we have (recall \eqref{eq:definition of total variation})
	\[
	\Var(\mathbbm{1}_E,\R^2)\le
	\sum_{j=1}^{\infty}\Var(\mathbbm{1}_{B(q_j,2^{-j})},\R^2)
	=2\pi\sum_{j=1}^{\infty}2^{-j}=2\pi.
	\]
	Thus $\mathbbm{1}_E\in\BV(\R^2)$.
	We have $M\mathbbm{1}_E(x)<1$ for every $x\in \R^2$ with
	$\mathbbm{1}_E^{\vee}(x)=0$. However,
	for every such $x$ there is a sequence
	of points $x_j\to x$ such that $x_j\in E$ for every $j\in\N$.
	Thus $M\mathbbm{1}_E(x_j)=1$ for every $j\in\N$.
	Hence $M\mathbbm{1}_E$
	is discontinuous at every point in the set
	$\{x\in\R^2\colon \mathbbm{1}_E^{\vee}(x)=0\}$,
	which has even infinite Lebesgue measure.
\end{example}

\section{Continuity and Lusin property on lines}\label{sec:continuity and Lusin}

In this section we prove that the non-centered maximal function of an SBV function,
when restricted to almost every line parallel to a coordinate axis,
is continuous and has the Lusin property. The Lusin property for a function
$v$ defined on $V\subset\R$ states that
\begin{equation}
\textrm{if }N\subset V\textrm{ with }\mathcal L^1(N)=0,\ \textrm{ then }
\mathcal L^1(v(N))=0.
\end{equation}

First we prove this kind of property in the following form.
Recall that $S_u$ denotes the set of non-Lebesgue points of $u$, and that $\widetilde{u}$
is the Lebesgue representative of $u$.

\begin{lemma}\label{lem:1d abs cont}
Let $V\subset\R$ be open and let $u\in\BV_{\loc}(V)$.
If $N\subset V\setminus S_u$ with $|Du|(N)=0$, then
\[
\mathcal L^1(\widetilde{u}(N))=0.
\]
\end{lemma}
\begin{proof}
The claim is equivalent with $\mathcal L^1(u^{\vee}(N))=0$;
we will work with the everywhere defined representative $u^{\vee}$
since some of the points that
we examine may be in the jump set $S_u$.
Fix $\eps>0$. We can take an open set $U$ with $N\subset U\subset V$
and $|Du|(U)<\eps$.
For every $x\in N$, we can choose an arbitrarily short compact interval $I\ni x$
contained in $U$. Consider the collection of intervals (understood to be
nondegenerate, i.e. consisting of more than one point)
\[
\mathcal I:=\{I\colon x\in I\subset U,\, x\in N\}.
\]
These form a covering of $N$.
Let
\[
H:=\{h\in\R\colon \textrm{for some }I\in\mathcal I,\ u^{\vee}(I)=\{h\}\}.
\]
The set $H$ can be at most countable,
since the intervals $I$ are nondegenerate.
Now the collection of intervals
\[
\mathcal J:=\big\{[\inf u^{\vee}(I) ,\,\sup u^{\vee}(I)]\colon I\in \mathcal I,
\,u^{\vee}\ \textrm{is not constant on }I\big\}
\]
is a covering of $u^{\vee}(N)\setminus H$. It is a fine covering, since every
$x\in N\subset V\setminus S_u$ is a point of continuity of $u^{\vee}$
(recall \eqref{eq:fundamental theorem of calculus for BV}).
By Vitali's covering theorem, there exists a countable collection of disjoint
intervals $\{J_j\}_{j=1}^{\infty}$ selected from $\mathcal J$
such that
\[
\mathcal L^1\Bigg(u^{\vee}(N)\setminus \bigcup_{j=1}^{\infty}J_j\Bigg)=0.
\]
For every $J_j$, there exists $I_j$ such that
$J_j=[\inf u^{\vee}(I_j),\sup u^{\vee}(I_j)]$. Then
by \eqref{eq:fundamental theorem of calculus for BV},
\[
\mathcal L^1(u^{\vee}(N))
\le \sum_{j=1}^{\infty}\mathcal L^1(J_j)
\le \sum_{j=1}^{\infty}|Du|(I_j)\le |Du|(U)<\eps.
\]
Since $\eps>0$ was arbitrary, the result follows.
\end{proof}

Define
\[
H_u:=\{x\in \Om\colon M_{\Om}u(x)>|u|^{\vee}(x)\}.
\]
Recall that we denote by $\pi\colon\R^d\to \R^{d-1}$ the orthogonal projection onto
$\R^{d-1}$: for $x=(x_1,\ldots,x_d)\in\R^d$,
\[
\pi((x_1,\ldots,x_d)):=(x_1,\ldots,x_{d-1}).
\] 

\begin{proposition}\label{prop:capacity and Hausdorff measure}
Let $A\subset \R^d$. Then
\[
2\mathcal H^{d-1}(\pi(A))\le \capa_1(A).
\]
\end{proposition}
\begin{proof}
Consider $u\in W^{1,1}(\R^d)$ with $u\ge 1$ in a neighborhood of $A$.
Then on almost every line $l$ in the $d$th coordinate direction intersecting
$A$, we have
\[
\int_{l}\left|\frac{du}{dx_d}\right|\,ds\ge 2.
\]
Integrating over $\R^{d-1}$, we get
\[
\int_{\R^d}|\nabla u|\,dx\ge 2\mathcal H^{d-1}(\pi(A)).
\]
Thus $\Vert u\Vert_{W^{1,1}(\R^d)}\ge 2\mathcal H^{d-1}(\pi(A))$ and
we get the result by taking infimum over all such $u$.
\end{proof}

Recall from Section \ref{sec:BV functions} that a BV function $u$ is in the SBV class
if $|D^c u|(\Om)=0$.

\begin{theorem}\label{thm:continuity on lines}
Let  $u\in\BV_{\loc}(\Om)\cap L^{1}(\Om)$ or $u\in \BV_{\loc}(\R^d)$.
Then $M_{\Om}u$ is continuous on almost every line parallel
to a coordinate axis.

If $u\in \SBV_{\loc}(\R^d)\cap L^{1}(\R^d)$
or $u\in \SBV_{\loc}(\R^d)\cap L^{\infty}(\R^d)$, then $Mu$
also has the Lusin property
on almost every line parallel to a coordinate axis.
\end{theorem}

\begin{proof}
Again we can assume that $u\ge 0$.
First we prove the continuity on lines.
Fix $\eps>0$.
By Theorems \ref{thm:quasisemicontinuity} and
\ref{thm:quasicontinuity of noncentered maximal function},
we can take an open set $G\subset\Om$ such that
$\capa_1(G)<\eps/C_2$ and $u^{\vee}|_{\Om\setminus G}$ is finite and upper semicontinuous,
and $M_{\Om}u|_{\Om\setminus G}$ is finite and continuous.
Since $\capa_1(S_u\setminus J_u)=\mathcal H^{d-1}(S_u\setminus J_u)=0$
(recall \eqref{eq:null sets of capa and Hausdorff}),
we can also assume that $G\supset S_u\setminus J_u$. 
By Lemma \ref{lem:uniform convergence of u} we can take an open set
$U\supset G$ such that $\capa_1(U)<\eps$ and
\[
\frac{1}{\mathcal L^d(B(x,r))}\int_{G\cap B(x,r)}u\,dy
\to 0\quad\textrm{as }r\to 0
\]
locally uniformly for $x\in \Om\setminus U$.
We wish to study the behavior of $M_{\Om}u$ in $H_u\setminus U$.
Consider $x_0\in H_u\setminus U$.
We have $\alpha:=M_{\Om}u(x_0)\in (0,\infty)$.
Let $\delta:=\alpha-u^{\vee}(x_0)>0$.
By upper semicontinuity of $u^{\vee}|_{\Om\setminus G}$,
for some $R_1>0$ we have
\begin{equation}\label{eq:upper semicontinuity with delta}
u^{\vee}(x)<\alpha-\frac{3\delta}{4}
\end{equation}
for all
$x\in B(x_0,R_1)\setminus G$.
By lower semicontinuity of the maximal function,
there exists $R_2>0$ such that
$M_{\Om}u(x)>\alpha-\delta/4$ for all $x\in B(x_0,R_2)$.
Moreover, by the choice of the set $U$, there exists $R_3>0$
such that for all $x\in B(x_0,R_3)\setminus U$ we have
\begin{equation}\label{eq:uniform smallness for G}
\frac{1}{\mathcal L^d(B(x,s))}\int_{G\cap B(x,s)}u\,dy
\le\frac{\delta}{2^{d+1}}
\quad\textrm{for all }0<s<R_3.
\end{equation}
Let $R:=\min\{R_1,R_2,R_3\}/4$.
Now consider any $x\in B(x_0,R)\setminus U$,
and any ball $B(z,r)\ni x$ with $r\in (0,R)$. We have
\begin{align*}
\frac{1}{\mathcal L^d(B(z,r))}\int_{B(z,r)}u\,dy
&= \frac{1}{\mathcal L^d(B(z,r))}\int_{B(z,r)\cap G}u\,dy
+\frac{1}{\mathcal L^d(B(z,r))}\int_{B(z,r)\setminus G}u\,dy\\
&\le \frac{2^d}{\mathcal L^d(B(x,2r))}\int_{B(x,2r)\cap G}u\,dy
+\alpha-\frac{3\delta}{4}\quad\textrm{by }\eqref{eq:upper semicontinuity with delta}\\
&\le\frac{\delta}{2}+\alpha-\frac{3\delta}{4}\quad\textrm{by }\eqref{eq:uniform smallness for G}\\
&=\alpha-\frac{\delta}{4}.
\end{align*}
On the other hand, we had $M_{\Om}u(x)>\alpha-\delta/4$ for all
$x\in B(x_0,R)$.
Recalling the definition \eqref{eq:auxiliary maximal operator}, we have
\begin{equation}\label{eq:maximal function is M R}
M_{\Om}u(x)=M_{\Om,R}u(x)
\end{equation}
for every $x\in B(x_0,R)\setminus U$.

Now we examine the behavior of $M_{\Om}u$ on lines.
Recall the notation and results from Section \ref{subsec:one dimensional sections}.
Without loss of generality we can
consider lines parallel to the $d$:th coordinate axis.
Recall that $\pi$ denotes the orthogonal projection onto $\R^{d-1}$.
By Proposition \ref{prop:capacity and Hausdorff measure} we know that
\[
\mathcal H^{d-1}(\pi(U))\le \capa_1(U)<\eps.
\]
Thus it is enough to consider a line not intersecting $U$.
In other words, consider a fixed $z\in \pi(\Om)\setminus \pi(U)$ and then consider
the line $(z,t)$, $t\in \R$.
Since we know that $M_{\Om}u|_{\Om\setminus U}$ is continuous,
$t\mapsto M_{\Om}u(z,t)$ is continuous, proving the first claim.

Now suppose $u\in \SBV_{\loc}(\R^d)\cap L^{1}(\R^d)$
or $u\in \SBV_{\loc}(\R^d)\cap L^{\infty}(\R^d)$.
The function $u_z$ is in
the class $\SBV_{\loc}(\R)$ for almost every $z\in \R^{d-1}$,
and so we can also assume that $u_z\in\SBV_{\loc}(\R)$.
We can further assume that $(J_u)_z$
is at most countable; this follows from the fact that $J_u$ is countably
$d-1$-rectifiable, and from the coarea formula,
see \cite[Theorem 2.93]{AFP}.

Let $N\subset \R$ with zero $1$-dimensional Lebesgue measure. We have 
\[
Mu(z,N)
=  M u((\{z\}\times N)\cap H_u)
\cup M u((\{z\}\times N)\setminus H_u).
\]
Here the first set has zero one-dimensional Lebesgue measure by
the fact that $Mu|_{H_u\setminus U}$ is locally Lipschitz, which is given by
\eqref{eq:maximal function is M R} and
Proposition \ref{prop:Lipschitz continuity}.
For the second set we have, recalling that $U\supset S_u\setminus J_u$,
\begin{align*}
M u((\{z\}\times N)\setminus H_u)
&\subset
M u((\{z\}\times N)\setminus (H_u\cup S_u))
\cup M u((\{z\}\times N)\cap J_u)\\
&= \widetilde{u}((\{z\}\times N)\setminus (H_u\cup S_u))
\cup M u((\{z\}\times N)\cap J_u)\\
&\subset \widetilde{u_z}(z,N\setminus S_{u_z(\cdot)})
\cup M u((\{z\}\times N)\cap J_u).
\end{align*}
by \eqref{eq:sections and jump sets} (which we can assume to hold by discarding another
$\mathcal L^{d-1}$-negligible set).
Here the first set has zero measure since $u_z\in \SBV_{\loc}(\Om_z)$ and
so $|Du_z|(N\setminus S_{u_z(\cdot)})=0$, and then we can use
Lemma \ref{lem:1d abs cont}. 
The second set has zero measure since $(\{z\}\times \R)\cap J_u$
was at most countable.
In total, $\mathcal L^1(Mu(z,N))=0$, and so $Mu$ has the Lusin property
on almost every line parallel to a coordinate axis.
\end{proof}

If we could give a positive answer to the following open problem, then
we could extend Theorem \ref{thm:continuity on lines} from SBV functions to BV functions.

\begin{openproblem}\label{openproblem about contact set}
Let $u\in\BV(\R^d)$. Is it true that $Mu(x)>|u|^{\vee}(x)$ for
$|D^c u|$-almost every $x\in\R^d$?
\end{openproblem}

In one dimension the answer is yes, see Proposition \ref{prop:Ddu outside Du is zero}.

\section{Sobolev property}\label{sec:Sobolev property}

In this section we show that if the non-centered maximal function of an SBV
function is locally BV, then it is in fact locally Sobolev.

We rely on the following direction of the Banach-Zarecki Theorem.
\begin{theorem}
Let $V\subset\R$ be open and let $v\in\BV_{\loc}(V)$ be a continuous function that satisfies the Lusin
property. Then $v$ is absolutely continuous on $V$.
\end{theorem}

We say that a function $v$ on $\R^d$ is ACL if it is absolutely continuous
on almost every line parallel to a coordinate axis.

The following theorem is Theorem \ref{thm:BV to Sobo intro}
in a slightly more general form.

\begin{theorem}\label{thm:BV to Sobo}
Let $u\in \SBV_{\loc}(\R^d)\cap L^{1}(\R^d)$
or $u\in \SBV_{\loc}(\R^d)\cap L^{\infty}(\R^d)$.
If $Mu\in \BV_{\loc}(\R^d)$, then
$Mu$ is ACL and in the class $W_{\loc}^{1,1}(\R^d)$. 
\end{theorem}

\begin{proof}
It is sufficient to consider the $d$:th coordinate direction.
Since $Mu\in\BV_{\loc}(\R^d)$, the function $Mu(z,\cdot)$ is in
the class $\BV_{\loc}(\R)$ for almost every $z\in \R^{d-1}$.
By Theorem \ref{thm:continuity on lines}, $Mu(z,\cdot)$ is also
continuous and has the Lusin property for almost every $z\in \R^{d-1}$.
Now by the Banach-Zarecki theorem,
$M u(z,\cdot)$ is absolutely continuous on $\R$ for almost every $z\in \R^{d-1}$.
In conclusion, $M u$ is ACL.
Since we already know that $Mu\in \BV_{\loc}(\R^d)$,
it follows that $Mu\in W_{\loc}^{1,1}(\R^d)$.
\end{proof}

In Theorem \ref{thm:BV to Sobo}, we would of course
like to prove that $Mu\in \BV_{\loc}(\R^d)$, instead of
assuming this. For sets of finite perimeter, such a better result is possible
due to a very recent result of Weigt \cite{Wei}.
We restate Theorem \ref{thm:perimeter to Sobo intro}, which is our third main theorem:

\begin{theorem}
Let $E\subset \R^d$ be a set of finite perimeter.
Then $M\mathbbm{1}_E\in W^{1,1}_{\loc}(\R^d)$ with
$\Vert \nabla M\mathbbm{1}_E\Vert_{L^1(\R^d)}\le C_d|D\mathbbm{1}_E|(\R^d)$,
where $C_d$ only depends on the dimension $d$.
\end{theorem}

\begin{proof}
Theorem 1.3 of \cite{Wei} states
that $M\mathbbm{1}_E\in \BV_{\loc}(\R^d)$
with $|DM\mathbbm{1}_E|(\R^d)\le C_d|D\mathbbm{1}_E|(\R^d)$.
Then by Theorem \ref{thm:BV to Sobo}, the conclusion follows.
\end{proof}

\section{Formula for the gradient}\label{sec:formula for derivative}

When $u\in\BV(\Om)$ and $M_{\Om}u\in W_{\loc}^{1,1}(\Om)$, there is a weak gradient
$DM_{\Om}u\in L_{\loc}^1(\Om)$. In this section we derive a formula for it,
generalizing \cite[Lemma 2.2(1)]{Lui} where the case $u\in W^{1,1}(\R^d)$
was considered.
However, as we will see, in the case $u\in\BV(\Om$)
the formula is, and can be, valid only under certain conditions.

As usual, $\Om\subset \R^d$ is an arbitrary nonempty open set.

We have the following standard approximation result for BV functions. 

\begin{proposition}\label{prop:approximation for BV}
Let $u\in\BV(\Om)$. Then there exists a sequence of functions $\{v_i\}_{i\in\N}$
in $C^{\infty}(\Om)$
such that $v_i\to u$ in $L^1(\Om)$,
$\lim_{i\to\infty}|D v_i|(\Om)=|Du|(\Om)$,
and $Dv_i\overset{*}{\rightharpoondown} Du$ and
$|Dv_i|\overset{*}{\rightharpoondown} |Du|$ in $\Om$.
\end{proposition}
\begin{proof}
By \cite[Theorem 3.9]{AFP} we find a sequence of functions $\{v_i\}_{i\in\N}$
in $C^{\infty}(\Om)$ such that $v_i\to u$ in $L^1(\Om)$ and
$\lim_{i\to\infty}|D v_i|(\Om)=|Du|(\Om)$.
Then by \cite[Proposition 3.13]{AFP} we have in fact $Dv_i\overset{*}{\rightharpoondown} Du$
in $\Om$,
and by \cite[Proposition 1.80]{AFP} also $|Dv_i|\overset{*}{\rightharpoondown} |Du|$ in $\Om$.
\end{proof}

We also have the following simple fact concerning the measures of spheres.

\begin{lemma}\label{lem:measure of spheres}
Let $\nu$ be a positive Radon measure on $\Om$ with
$\nu(\Om)<\infty$ and $\nu\ll \mathcal H^{d-1}$.
Then $\nu(\partial B)>0$ for at most countably many spheres $\partial B$ with
$\overline{B}\subset \Om$.
\end{lemma}

\begin{proof}
For any distinct $y,z\in\Om$ with $\overline{B}(y,r)\subset \Om$ and
$\overline{B}(z,R)\subset \Om$ for some $r,R>0$,
the intersection of the spheres $\partial B(y,r)$ and
$\partial B(z,R)$ has zero $\mathcal H^{d-1}$-measure and thus zero $\nu$-measure.
Thus for every $\alpha>0$, there can be only finitely many balls
$\overline{B}\subset \Om$ such that $\nu(\partial B)>\alpha$.
The result follows.
\end{proof}

Given a ball $B=B(x,r)$ (open, as usual) and $k\in\{1,\ldots,d\}$,
define the ``half-open'' balls
\[
\overline{B}^{k,+}:=B\cup\{y\in \partial B\colon y_k> x_k\},\quad
\overline{B}^{k,-}:=B\cup\{y\in \partial B\colon y_k< x_k\}.
\]

The following definition of ``optimal balls'' is convenient when studying the non-centered
maximal function, and has been used e.g. in \cite{Lui}.
Recall also Lemma \ref{lem:point in closure of ball}.

\begin{definition}
For a function $u$ on $\Om$ and $x\in\Om$, let
\[
\mathcal B_x:=\left\{B(z,r)\subset \Om\colon x\in \overline{B}(z,r),\,r>0,
\textrm{ and }\vint{B(z,r)}|u|\,dy
=M_{\Om}u(x)\right\}.
\]
\end{definition}

Note that if $u\in L^1(\Om)$ and
$x\in \Om$ is a Lebesgue point of $u$ with  $\mathcal B_x=\emptyset$, then we necessarily have
$M_{\Om}u(x)=\widetilde{u}(x)$.

The following is our first version of a formula for $DM_{\Om}u$.
Note that when $M_{\Om}u\in\BV_{\loc}(\Om)$, at almost every $x\in\Om$
we can interpret $D_k M_{\Om}u(x)$, $k\in\{1,\ldots,d\}$,
to be either the density of the measure $D_k M_{\Om}u$ or the classical
partial derivative of $M_{\Om}u$
(for the classical partial derivative to make sense, we need the correct pointwise
representative of $M_{\Om}u$, but the first part of Theorem \ref{thm:continuity on lines}
guarantees that $M_{\Om}u$ itself is suitable).
The same applies to
$D_k\widetilde{u}$.

\begin{theorem}\label{thm:derivative inequalities}
Let $u\in\BV(\Om)$ such that $M_{\Om}u\in \BV_{\loc}(\Om)$.
Then for almost every $x\in\Om$ and
every $k\in\{1,\ldots,d\}$, we have
\begin{align*}
&\mathrm{(1)}\ \frac{D_k|u|({\overline{B}^{k,+})}}{\mathcal L^d(B)}
\le D_k M_{\Om}u(x)\le \frac{D_k|u|(\overline{B}^{k,-})}{\mathcal L^d(B)}\quad
\textrm{if }B\in\mathcal B_x,\, \overline{B}\subset\Om,\\
&\mathrm{(2)}\ D_k M_{\Om}u(x)=D_k|\widetilde{u}|(x)\quad \textrm{if }\mathcal B_x=\emptyset.
\end{align*}
\end{theorem}
\begin{proof}
As usual, we can assume that $u\ge 0$.
%By Theorem \ref{thm:BV to Sobo}, we have $M_{\Om}u\in W_{\loc}^{1,1}(\Om)$.
Take $x\in\Om$ such that all $D_k M_{\Om}u(x)$, $k=1,\ldots,d$,
exist (both as densities and as classical derivatives).

First suppose that $B(z,r)\in\mathcal B_x$ and
$\overline{B(z,r)}\subset \Om$.
Fix $k\in\{1,\ldots,d\}$.
Now $B(z+he_k,r)\subset\Om$ for $h\in\R$ close to zero.

Consider momentarily $v\in C^{\infty}(\Om)$. We get for small $h>0$ that
\begin{equation}\label{eq:version for smooth functions}
\begin{split}
\frac{1}{h}\left(\,\vint{B(z+he_k,r)}v\,dy-\vint{B(z,r)}v\,dy\right)
&=\frac{1}{h}\left(\,\vint{B(z,r)}v(y+he_k)-v(y)\,dy\right)\\
&=\frac{1}{h}\left(\,\vint{B(z,r)}\int_0^hD_k v(y+te_k)\,dt\,dy\right)\\
&= \frac{1}{h}\int_0^h\left(\,\vint{B(z+te_k,r)}D_k v\,dy\right)\,dt.
\end{split}
\end{equation}
By Proposition \ref{prop:approximation for BV}
we find a sequence $\{v_i\}_{i\in\N}$ in $C^{\infty}(\Om)$ such that $v_i\to u$
in $L^1(\Om)$, $\lim_{i\to\infty}|D v_i|(\Om)=|Du|(\Om)$,
and $Dv_i\overset{*}{\rightharpoondown} Du$
and $|Dv_i|\overset{*}{\rightharpoondown} |Du|$ in $\Om$.
For every ball $B$ with $\overline{B}\subset \Om$ and
$|Du|(\partial B)=0$, this implies (see \cite[Proposition 1.62(b)]{AFP})
\[
D v_i(B)\to D u(B)\quad\textrm{so in particular}\quad D_k v_i(B)\to D_k u(B).
\]
By Lemma \ref{lem:measure of spheres}, there are at most countably many spheres
$\partial B$
with $\overline{B}\subset \Om$ and $|Du|(\partial B)>0$.
In particular, $|D u|(\partial B(z+te_k,r))=0$ for almost every $t\in [0,h]$.
Writing \eqref{eq:version for smooth functions} with $v=v_i$ and taking the limit
$i\to\infty$, we get by Lebesgue's dominated convergence theorem
\begin{align*}
&\frac{1}{h}\left(\,\vint{B(z+he_k,r)}u\,dy
-\vint{B(z,r)}u\,dy\right)
= \frac{1}{h}\int_0^h\left(\,\frac{D_k u(B(z+te_k,r))}{\mathcal L^d(B(z,r))}\right)\,dt.
\end{align*}
Thus
\begin{align*}
D_k M_{\Om}u(x)
&=\lim_{h\to 0^+}\frac{1}{h}\left(M_{\Om}u(x+he_k)-M_{\Om}u(x)\right)\\
&\ge \lim_{h\to 0^+}\frac{1}{h}\left(\,\vint{B(z+he_k,r)}u\,dy
-\vint{B(z,r)}u\,dy\right)\\
&=\lim_{h\to 0^+}\frac{1}{h}\int_0^h\left(\,\frac{D_k u(B(z+te_k,r))}
{\mathcal L^d(B(z,r))}\right)\,dt.
\end{align*}
Here
\begin{align*}
&\left|\frac{1}{h}\int_0^h\left(\,\frac{D_k u(B(z+te_k,r))}
{\mathcal L^d(B(z,r))}\right)\,dt
-\frac{D_k u(\overline{B}^{k,+}(z,r))}{\mathcal L^d(B(z,r))}\right|\\
&\qquad\le  \frac{1}{h}\int_0^h\left|\,\frac{D_k u(B(z+te_k,r))}
{\mathcal L^d(B(z,r))}
-\frac{D_k u(\overline{B}^{k,+}(z,r))}{\mathcal L^d(B(z,r))}\right|\,dt\\
&\qquad\le  \frac{1}{\mathcal L^d(B(z,r))}\frac{1}{h}\int_0^h
|D u|\left(\overline{B}^{k,+}(z,r)\Delta \bigcup_{s\in (0,h)}B(z+se_k,r)\right)\,dt\\
&\qquad=  \frac{1}{\mathcal L^d(B(z,r))}
|D u|\left(\overline{B}^{k,+}(z,r)\Delta \bigcup_{s\in (0,h)}B(z+se_k,r)\right)\\
&\qquad\to 0\quad\textrm{as }h\to 0
\end{align*}
since $\overline{B}^{k,+}(z,r)\Delta \bigcup_{s\in (0,h)}B(z+se_k,r)\to \emptyset$.
Combining the previous inequalities, we get
\[
D_k M_{\Om}u(x)\ge \frac{D_k u(\overline{B}^{k,+}(z,r))}{\mathcal L^d(B(z,r))}.
\]
Similarly we get
\begin{align*}
D_k M_{\Om}u(x)
&=\lim_{h\to 0^+}\frac{1}{h}\left(M_{\Om}u(x)-M_{\Om}u(x-he_k)\right)\\
&\le \lim_{h\to 0^+}\frac{1}{h}\left(\,\vint{B(z,r)}u\,dy
-\vint{B(z-he_k,r)}u\,dy\right)\\
&=\frac{D_k u(\overline{B}^{k,-}(z,r))}{\mathcal L^d(B(z,r))}.
\end{align*}

Then assume that $\mathcal B_x=\emptyset$. Discarding another $\mathcal L^d$-negligible set,
we can assume that $x$ is a Lebesgue point of $u$, and so necessarily
$M_{\Om}u(x)=\widetilde{u}(x)$.
We can also assume that $D_k\widetilde{u}(x)$, $k=1,\ldots,d$, exist
(again, both as densities and as partial derivatives). Thus
\begin{align*}
D_k M_{\Om}u(x)
&=\lim_{h\to 0^+}\frac{1}{h}\left(M_{\Om}u(x+he_k)-M_{\Om}u(x)\right)\\
&\ge \lim_{h\to 0^+}\frac{1}{h}\left(\widetilde{u}(x+he_k)-\widetilde{u}(x)\right)\\
&=D_k \widetilde{u}(x).
\end{align*}
Similarly,
\begin{align*}
D_k M_{\Om}u(x)
&=\lim_{h\to 0^+}\frac{1}{h}\left(M_{\Om}u(x)-M_{\Om}u(x-he_k)\right)\\
&\le \lim_{h\to 0^+}\frac{1}{h}\left(\widetilde{u}(x+he_k)-\widetilde{u}(x)\right)\\
&=D_k \widetilde{u}(x).
\end{align*}
\end{proof}

From the viewpoint of having a formula for $DM_{\Om}u$,
we would of course like to have equality in Theorem
\ref{thm:derivative inequalities}(1), which in particular happens if $|Du|(\partial B)=0$.
With this in mind, we prove the following lemma.

\begin{lemma}\label{lem:intersections with spheres}
Let $d\ge 2$ and let $u\in\BV_{\loc}(\Om)$.
For $\mathcal L^d$-almost every $x\in\Om$, we have
\[
|Du|(\partial B)=0
\]
for every ball $B$ with $\overline{B}\subset\Om$ and $x\in\partial B$.
\end{lemma}
\begin{proof}
We can apply Lemma \ref{lem:measure of spheres} with the choice
$\nu=|Du|$ to obtain that there are at most countably
many spheres $\partial B$ such that $|Du|(\partial B)>0$.
Thus $|Du|(\partial B)>0$ and $x\in\partial B$
can only be true if $x$ belongs to the countable union of spheres, and such a union
of course has Lebesgue measure zero.
\end{proof}

The following fact is easy to prove, see \cite[Lemma 2.2(2)]{Lui}.
\begin{lemma}\label{lem:optimal ball contains point}
Let $u$ be a function on $\Om$ and let
$x\in \Om$. If there is a ball $B\in\mathcal B_x$ with
$x\in B$, then $DM_{\Om}u(x)=0$ (as a classical derivative).
\end{lemma}

Now we can prove the following formula for the gradient of the maximal function.

\begin{theorem}\label{thm:derivative formula}
Let $u\in\BV(\Om)$ with $M_{\Om}u\in \BV_{\loc}(\Om)$. Then for almost every $x\in\Om$,
\begin{align*}
&\mathrm{(1)}\ D M_{\Om}u(x)=\frac{D|u|(B)}{\mathcal L^d(B)}\quad
\textrm{if }B\in\mathcal B_x,\ \overline{B}\subset \Om,\
D M_{\Om}u(x)\neq 0,\textrm{ and }d\ge 2,\\
&\mathrm{(2)}\ D M_{\Om}u(x)=D|\widetilde{u}|(x)\quad \textrm{if }\mathcal B_x=\emptyset.
\end{align*}
\end{theorem}

In Examples \ref{ex:1d example} and
\ref{ex:derivative counterexample} we will show that the assumptions
$\overline{B}\subset \Om$, $D M_{\Om}u(x)\neq 0$, and $d\ge 2$ are needed.

\begin{proof}
First suppose that $B\in\mathcal B_x$, $\overline{B}\subset \Om$,
$D M_{\Om}u(x)\neq 0$, and $d\ge 2$.
Note that since $DM_{\Om} u(x)\neq 0$, we have
$x\in \partial B$ by Lemma \ref{lem:optimal ball contains point}.
Then by Lemma \ref{lem:intersections with spheres} we can assume that
$|Du|(\partial B)=0$
(recall \eqref{eq:Dju and Dcu}).
Thus by Theorem \ref{thm:derivative inequalities}(1), we get (1).

If $\mathcal B_x=\emptyset$, Theorem \ref{thm:derivative inequalities}(2) gives (2).
\end{proof}

\begin{example}\label{ex:1d example}
On the real line, take $\Om=\R$ and
\[
u(x):=
\begin{cases}
x & \textrm{when }0\le x\le 1,\\
0 & \textrm{otherwise.}
\end{cases}
\]
Thus $u\in\BV(\R)$.
Now obviously for every $x\in [0,1]$ we have $\mathcal B_x=\{(x,1)\}$, so that
\[
M u(x)=\vint{(x,1)}u(t)\,dt=\frac{1}{2}(1+x).
\]
Hence $DMu(x)=1/2$ for all $x\in (0,1)$, but
\[
\frac{Du((x,1))}{\mathcal L^1((x,1))}=\frac{Du([x,1))}{\mathcal L^1((x,1))}=1\quad 
\textrm{and}\quad
\frac{Du((x,1])}{\mathcal L^1((x,1))}=\frac{Du([x,1])}{\mathcal L^1((x,1))}=1-\frac{1}{1-x}.
\]
Thus Theorem \ref{thm:derivative formula}(1) fails, also if $B$ is replaced by any
half-open interval. Hence the assumption $d\ge 2$ is necessary.

A small modification of this example shows that the assumption $\overline{B}\subset\Om$
is also needed. On the plane, let $\Om=(0,1)\times (-2,2)$ and
\[
u(x_1,x_2):=x_1.
\]
Now obviously for every $x\in (0,1)\times (-1,1)$,
we have $\mathcal B_x=\{B_x\}$ with
\[
B_x:=\left\{B\Bigg(\Big(\frac{1+x_1}{2},x_2\Big),\frac{1-x_1}{2}\Bigg)\right\}, 
\]
so that
\[
M_{\Om} u(x)=\vint{B_x}u\,dy=\frac{1}{2}(1+x_1).
\]
Hence $D_1 M_{\Om}u(x)=1/2$ for all $x\in (0,1)\times (-1,1)$, but
\[
\frac{D_1u(B_x)}{\mathcal L^2(B_x)}=1,
\]
showing that the assumption $\overline{B}\subset\Om$ is necessary.
Moreover $|Du|(\partial B_x\setminus \partial \Om)=0$, so again including any part of the boundary
of $B_x$ would not help either.
\end{example}

Finally we give an example showing that the assumption
$D M_{\Om}u(x)\neq 0$ is also needed in Theorem \ref{thm:derivative formula}.

\begin{example}\label{ex:derivative counterexample}
On the plane, let $\Om=\R^2$ and
\[
E_0:=B(0,1)\setminus B(0,1-\delta)
\]
for a small $\delta$; choose $\delta=0.01$.
Then for every $x\in B(0,\delta)$, we claim that $\mathcal B_x^{E_0}=\{B(0,1)\}$;
we use the superscript $E_0$ to specify that we consider the collection of optimal
balls with respect to the function $\mathbbm{1}_{E_0}$.
To see this, fix $x\in B(0,\delta)$ and $B(z,r)\in \mathcal B_x^{E_0}$ (such a disk is easily seen
to exist). Note first that
necessarily $0.49\le r\le 1$.
Now we simply check three cases.

First suppose $0.49\le r\le 0.63$.
Now $B(z,r)$ intersects less than $22/100$ of $\partial B(0,1-\delta)$, and so
\begin{align*}
m(B(z,r))
&:=\frac{\mathcal L^2(E_0\cap B(z,r))}{\mathcal L^2(B(z,r))}
\le \frac{\mathcal L^2(E_0\cap B(z,0.63))}{\mathcal L^2(B(z,0.49))}\\
&<\frac{\tfrac{22}{100} \times 2\pi\delta}{0.49^2 \mathcal L^2(B(0,1))}
< \frac{0.99\times 2\pi\delta}{\mathcal L^2(B(0,1))}
<\frac{\mathcal L^2(E_0)}{\mathcal L^2(B(0,1))}=m(B(0,1)).
\end{align*}
Then suppose $0.63\le r\le 0.8$.
Now $B(z,r)$ intersects
less than $1/3$ of $\partial B(0,1-\delta)$, and so
\begin{align*}
m(B(z,r))
<\frac{\tfrac 13 \times 2\pi\delta}{0.63^2 \mathcal L^2(B(0,1))}
<m(B(0,1)).
\end{align*}
Finally suppose $0.8\le r\le 0.95$.
Now $B(z,r)$ intersects
less than $1/2$ of $\partial B(0,1-\delta)$, and so
\begin{align*}
m(B(z,r))
<\frac{\tfrac 12 \times 2\pi\delta}{0.8^2 \mathcal L^2(B(0,1))}
<m(B(0,1)).
\end{align*}

Thus necessarily $0.95\le r\le 1$.
But now $B(z,r)\supset \partial B(0,1-\delta)$,
because otherwise $B(z,r)$ covers less than $3/4$ of $\partial B(0,1-\delta/2)$,
implying that
\[
m(B(z,r))<\frac{\tfrac 78 \times 2\pi\delta}{0.95^2 \mathcal L^2(B(0,1))}<m(B(0,1)).
\]
Thus $B(z,r)$ contains $B(0,1-\delta)$, and now clearly the maximum value of $m(B(z,r))$
is obtained by choosing $B(z,r)=B(0,1)$.
Thus $\mathcal B_x^{E_0}=\{B(0,1)\}$ as desired.

Denote $c:=(0,1)\in\partial B(0,1)$.
Next we ``perturb'' $E_0$ slightly by removing and adding a small ball:
\[
E:=(E_0\setminus B(c,\delta^2))\cup (B(c,\delta^2)\setminus B(0,1)).
\]
Since the perturbation is so small, almost the same calculations as above show that
for every $x\in B(0,\delta)$, we have $\mathcal B_x^{E}=\{B(0,1)\}$.
Thus $DM\mathbbm{1}_{E}=0$ in $B(0,\delta)$, but
\[
\frac{D_2 \mathbbm{1}_{E}(B(0,1))}{\mathcal L^2(B(0,1))}<0,
\quad \frac{D_2 \mathbbm{1}_{E}(\overline{B}(0,1))}{\mathcal L^2(B(0,1))}>0,
\]
\[
\frac{D_2 \mathbbm{1}_{E}(\overline{B}^{2,+}(0,1))}{\mathcal L^2(B(0,1))}<0,
\quad \frac{D_2 \mathbbm{1}_{E}(\overline{B}^{2,-}(0,1))}{\mathcal L^2(B(0,1))}>0.
\]
Thus none of these equal $DM\mathbbm{1}_{E}$ in $B(0,\delta)$, which is of course a set
of nonzero Lebesgue measure.
\end{example}

\section{The one-dimensional case}\label{sec:the 1d case}

In this section we investigate the properties of the non-centered maximal function
in the special case $d=1$.
Let $\Om\subset \R$ be an arbitrary nonempty open set.

Aldaz and P\'erez L\'azaro \cite{APL07} proved in one dimension that the non-centered
maximal function of a function $u\in\BV(\Om)$ is in the Sobolev class
$W^{1,1}_{\loc}(\Om)$, with $\Vert D M_{\Om}u\Vert_{L^1(\Om)}\le |Du|(\Om)$.
Now we investigate the behavior of $M_{\Om}u$ a little further.

Recall that
\[
H_u=\{x\in \Om\colon M_{\Om}u(x)>|u|^{\vee}(x)\}.
\]
Recall also from the proof of Theorem \ref{thm:continuity on lines} and the Open Problem on
Page \pageref{openproblem about contact set} that
it is in some sense desirable that the set $H_u$ be as large as possible.
On the real line, we are able to show the following.
Recall that $S_u$ is the set of non-Lebesgue points, or (as we are in one dimension)
the set of discontinuity points of $u^{\vee}$ (alternatively some other good pointwise representative).
Moreover, denote by $\partial^*E$ the measure-theoretic boundary of a set $E\subset \R$,
i.e. the set of points $x\in\R$ for which
\[
\limsup_{r\to 0}\frac{\mathcal L^1(B(x,r)\cap E)}{\mathcal L^1(B(x,r))}>0
\quad\textrm{and}\quad \limsup_{r\to 0}\frac{\mathcal L^1(B(x,r)\setminus E)}{\mathcal L^1(B(x,r))}>0.
\]

\begin{proposition}\label{prop:Ddu outside Du is zero}
Let $u\in\BV_{\loc}(\Om)$ with $|Du|(\Om)<\infty$. Then $|D u|(\Om\setminus (H_u\cup S_u))=0$.
\end{proposition}
\begin{proof}
Abbreviate super-level sets by
\[
\{u>t\}:=\{x\in\Om\colon u(x)>t\},\quad t\in\R.
\]
We have the coarea formula (see e.g. \cite[Theorem 3.40]{AFP})
\begin{equation}\label{eq:coarea formula}
|Du|(A)=\int_{-\infty}^{\infty}|D\mathbbm{1}_{\{u>t\}}|(A)\,dt
\end{equation}
for every Borel set $A\subset\Om$.
In particular, for almost every $t\in\R$, the super-level set $\{u>t\}$
has finite perimeter in $\Om$.
Let $N\subset\R$ be the exceptional set.
For later purposes, we also include in $N$ the at most countably many $t\in\R$ for which
$\mathcal L^1(\{u=t\})>0$.

Consider $t\in\R\setminus N$.
First assume that $t\ge 0$.
The fact that $\{u>t\}$ has finite perimeter in $\Om$ means that
after redefinition in a set of measure zero, giving a set that we can denote $E_t$,
the relative boundary $\partial E_t\cap\Om\subset \Om$
consists of finitely many points; see \cite[Proposition 3.52]{AFP}.
It follows that
if $x\in\partial E_t\cap\Om$, then for small enough
$r>0$, necessarily $\mathcal L^1((x-r,x)\setminus\{u\le t\})=0$ and
$\mathcal L^1((x,x+r)\setminus \{u>t\})=0$, or vice versa.
Supposing without loss of generality the former, we have
\[
M_{\Om}u(x)\ge \vint{(x,x+r)}u\,dy>t.
\]
If also $x\notin S_u$, then $t=u^{\vee}(x)=|u|^{\vee}(x)$, and we get $x\in H_u$.
In conclusion, if
$x\in \partial^*\{u>t\}\cap\Om\setminus S_u\subset \partial E_t\cap\Om\setminus S_u$
for some $t\in\R\setminus N$,
then $x\in H_u$.

For $t<0$, we similarly get that if $x\in \partial E_t\cap\Om$ for $t\in\R\setminus N$, then
necessarily $\mathcal L^1((x-r,x)\setminus\{u< t\})=0$ and
$\mathcal L^1((x,x+r)\setminus \{u>t\})=0$, or vice versa
(recall that $\mathcal L^1(\{u=t\})=0$). Supposing the former,
\[
M_{\Om}u(x)\ge \vint{(x-r,x)}|u|\,dy>|t|.
\]
If $x\notin S_u$, then $|t|=|u|^{\vee}(x)$ and we get $x\in H_u$, as before.

Consider a point $x$ in the Borel set $A:=\Om\setminus (H_u\cup S_u)$.
Now $x\notin\partial^*\{u>t\}$ for all $t\in\R\setminus N$.
For every $t\in\R\setminus N$ we also have
\begin{equation}\label{eq:perimeter and Hausdorff}
|D\mathbbm{1}_{\{u>t\}}|(A)=\mathcal H^0(\partial^*\{u>t\}\cap A),
\end{equation}
see e.g. \cite[Theorems 3.59 \& 3.61]{AFP}.
Thus by \eqref{eq:coarea formula} we get
\[
|Du|(A)=\int_{-\infty}^{\infty}\mathcal H^0(\partial^*\{u>t\}\cap A)\,dt
=\int_{N}\mathcal H^0(\partial^*\{u>t\}\cap A)\,dt=0.
\]
\end{proof}

Since $M_{\Om}u$ is continuous (even absolutely) and $u^{\vee}$ is upper semicontinuous,
$H_u$ is an open set, so it is the union of disjoint open
intervals 
\[
H_u=\bigcup_{j=1}^{\infty}(a_j,b_j)\subset \Om.
\]
Note that for two of these intervals, $a_j$ or $b_j$ may be $\pm \infty$.

\begin{theorem}
Let $u\in\BV_{\loc}(\Om)$ with $|Du|(\Om)<\infty$.
For each $j\in\N$ there exists a point $c_j\in (a_j,b_j)$ such that
we have the representation
\[
|DM_{\Om}u|(\Om)=
\sum_{j=1}^{\infty}
\big[|M_{\Om}u(a_j)-M_{\Om}u(c_j)|+|M_{\Om}u(c_j)-M_{\Om}u(b_j)|\big].
\]
\end{theorem}
If $a_j$ (or $b_j$) is $\pm \infty$, we interpret 
$M_{\Om}u(a_j)$ as a limit; it exists since
$M_{\Om}u\in W^{1,1}_{\loc}(\Om)$ with $\Vert D M_{\Om}u\Vert_{L^1(\Om)}\le |Du|(\Om)$
by \cite[Theorem 2.5]{APL07}.
\begin{proof}
The set $S_u$ is at most countable and $M_{\Om}u\in W_{\loc}^{1,1}(\Om)$, so we have
\begin{equation}\label{eq:variation in Su}
|DM_{\Om}u|(S_u)=0.
\end{equation}
Consider a point $x\in \Om\setminus (S_u\cup H_u)$.
Now $M_{\Om}u(x)=|u|^{\vee}(x)$. Since both functions
$M_{\Om}u$ and $|u|^{\vee}$ are continuous at
$x$, this point can be in $\partial^*\{M_{\Om}u-|u|>t\}$ only for $t=0$.
By the coarea formula \eqref{eq:coarea formula} and \eqref{eq:perimeter and Hausdorff}, we have
\begin{align*}
|D(M_{\Om}u-|u|)|(\Om\setminus (S_u\cup H_u))
&=\int_{-\infty}^{\infty}\mathcal H^0(\Om\cap \partial^*\{M_{\Om}u-|u|>t\}
\setminus (S_u\cup H_u))\,dt\\
&=0.
\end{align*}
By Proposition \ref{prop:Ddu outside Du is zero}, it follows that
\[
|DM_{\Om}u|(\Om\setminus (S_u\cup H_u))
=|D|u||(\Om\setminus (S_u\cup H_u))
\le|Du|(\Om\setminus (S_u\cup H_u))
=0.
\]
Thus by \eqref{eq:variation in Su},
\[
|DM_{\Om}u|(\Om\setminus H_u)=0.
\]
Thus all of the total variation of $M_{\Om}u$ is in the set
$H_u=\bigcup_{j=1}^{\infty}(a_j,b_j)$.

Now we follow an argument given in \cite{APL07}.
Suppose that for some $j\in\N$,
there exist points $d_1,d_2,d_3$ with $a_j<d_1<d_2<d_3<b_j$
and $M_{\Om}u(d_1)<M_{\Om}u(d_2)$ and $M_{\Om}u(d_3)<M_{\Om}u(d_2)$.
We can assume that $M_{\Om}u(d_2)=\max\{M_{\Om}u(x)\colon x\in[d_1,d_3]\}$.
Then by \cite[Lemma 3.6]{APL07} we have $M_{\Om}u(d_2)=|u|^{\vee}(d_2)$,
a contradiction with $d_2\in H_u$.

It follows that for every $j\in\N$, either $M_{\Om}u$
is monotone on $(a_j,b_j)$ or there exists $c_j\in (a_j,b_j)$
such that $M_{\Om}u$ is decreasing on $[a_j,c_j]$ and increasing on $[c_j,b_j]$.
In the former case, we can just choose an arbitrary $c_j\in (a_j,b_j)$.
Now
\[
|DM_{\Om}u|(\Om)=
|DM_{\Om}u|(H_u)=\sum_{j=1}^{\infty}
\big[|M_{\Om}u(a_j)-M_{\Om}u(c_j)|+|M_{\Om}u(c_j)-M_{\Om}u(b_j)|\big].
\]
\end{proof}

\noindent Address:\\

\noindent Academy of Mathematics and Systems Science,\\
Chinese Academy of Sciences,\\
Beijing 100190, PR China\\
E-mail: {\tt panu.lahti@aalto.fi}

\end{document}